\documentclass[11pt,reqno,tbtags]{amsart}
\usepackage[]{amsmath,amssymb,amsfonts,latexsym,amsthm,enumerate}
\usepackage{bbm}
\usepackage{enumerate}
\usepackage[numeric,initials,nobysame]{amsrefs}

\numberwithin{equation}{section}

\newtheorem{maintheorem}{Theorem}

\newtheorem{theorem}{Theorem}[section]
\newtheorem*{theorem*}{Theorem}

\newtheorem{proposition}[theorem]{Proposition}

\theoremstyle{definition}{

\newtheorem*{definition*}{Definition}

}

\theoremstyle{remark}{
\newtheorem*{remark*}{Remark}

}

\newcommand{\Z}{\mathbb Z}

\renewcommand{\P}{\mathbb{P}}

\newcommand{\gap}{\lambda}

\newcommand{\tmix}{t_\textsc{mix}}
\newcommand{\tv}{{\textsc{tv}}}

\renewcommand{\epsilon}{\varepsilon}
\renewcommand{\phi}{\varphi}

\date{}

\begin{document}
\title{Mixing of the upper triangular matrix walk}

\author{Yuval Peres}
\address{Yuval Peres\hfill\break
Microsoft Research\\
One Microsoft Way\\
Redmond, WA 98052-6399, USA.}
\email{peres@microsoft.com}
\urladdr{}

\author{Allan Sly}
\address{Allan Sly\hfill\break
Microsoft Research\\
One Microsoft Way\\
Redmond, WA 98052-6399, USA.}
\email{allansly@microsoft.com}
\urladdr{}

\begin{abstract}

We study a natural random walk over the upper triangular matrices, with entries in the field $\Z_2$, generated by steps which add row~$i+1$ to row $i$.  We show that the mixing time of the lazy random walk is $O(n^2)$ which is optimal up to constants.  Our proof makes key use of the linear structure of the group and extends to walks on the upper triangular matrices over the fields $\Z_q$ for $q$ prime.

\end{abstract}

\maketitle

\section{Introduction}

For $n\geq 2$ and $q$ a prime let $G_{n}(q)$ be the group of $n\times n$-upper triangular matrices under multiplication with entries in the field $\Z_q$ and 1 along the diagonal.  The upper triangular walk on $G_n(2)$ is the Markov chain where each step is given by choosing a random $1\leq i \leq n-1$ and adding row $i+1$ to row $i$.  The walk is stationary with respect to the uniform distribution on $G_n(2)$ and our main result establishes the order of the mixing time answering a question of Stong~\cite{Stong:95} and of  Arias-Castro, Diaconis and Stanley~\cite{ADS:04} and Problem~14 of~\cite{LPW}.

\begin{maintheorem}\label{t:main}
There exists a constant $C>0$ such that the mixing time of the lazy upper triangular walk  satisfies $\tmix^{(n)} \leq C n^2$.
\end{maintheorem}
This bound is tight up to constants as it is known that the mixing time is at least order $n^2$ (see e.g. \cite{Stong:95}).  Indeed the projection onto the rightmost column of the matrix is itself a Markov chain given by an East model (see Section~\ref{s:East}) which has mixing time of order $n^2$.

\subsection{Background}

The random walks on $G_n(q)$ have received significant attention.  There are two natural extensions to fields $\Z_q$ for $q$  prime.  We will focus on the walk considered by Coppersmith and Pak~\cite{Pak:00,CopPak:00}.
This walk is given by the set of generators $\{I+ a E_{i,i+1}:1\leq i \leq n-1,a\in \Z_q\}$ where $E_{i,j}$ denotes the $n\times n$ matrix with 1 at position $(i,j)$ and 0 everywhere else.  We denote this walk as $\mathcal{W}_1$.   An equivalent description is that each step entails choosing a uniformly random $1\leq i \leq n-1$ and $a\in \Z_q$ and adding row $i+1$ multiplied by $a$ to row $i$.  This is an ergodic Markov chain reversible with respect to the uniform distribution on $G_n(q)$.  When $q=2$ this corresponds to the lazy version of our upper triangular walk of Theorem~\ref{t:main}.

The other extension, which we will denote $\mathcal{W}_2$, is the walk on $G_n(q)$ given by generators $\{I \pm E_{i,i+1}:1\leq i \leq n-1\}$ which has been studied by Diaconis, Saloff-Coste, Stong and others~\cite{DiaSal:94,Stong:95,ADS:04}.  In the case $q=2$ it is exactly the upper triangular walk of Theorem~\ref{t:main}.
Random walks of the form $\mathcal{W}_2$ were first studied by Zack~\cite{Zack:84} in the case $n=3$ as a random walk on the Heisenberg Group.  In the setting of $q$ growing and $n$ fixed,  Diaconis and Saloff-Coste~\cite{DiaSal:94,DiaSal:95,DiaSal:96} introduced a number of sophisticated techniques giving sharp rates of convergence.

In the case of $n$ growing the first results follow from work of Ellenberg~\cite{Ellenberg:93} which gives an upper bound of order $n^7$.  This was subsequently improved by Stong who gave sharp bounds on the spectral gap of $\Theta(n^{-1}q^{-2})$ which translates into an upper bound on the mixing time of $n^3 q^2\log q$.

The use of character theory has proved to be a key tool in analyzing random walks on groups (see e.g. \cite{DiaSha:81,Diaconis:03}).
One reason random walks on $G_n(q)$ have resisted analysis is the lack of understanding of the character theory of $G_n(q)$ which is ``considered unknowable''~\cite{ADS:04}.  Arias-Castro, Diaconis and Stanley~\cite{ADS:04} approached the problem  using the super-character theory of Andre, Carter and Yan and gave a sharp analysis of a related chain whose generators are the conjugacy classes of $I \pm E_{i,i+1}$.  Using comparison, this approach yielded an upper bound on the mixing time of $O(n^4 q^2 \log(n)\log(q))$ for the chain $\mathcal{W}_2$.

Using a stopping time argument, Pak~\cite{Pak:00} showed that when $q>2n^3$, the mixing time of the walk $\mathcal{W}_1$ is $O(n^{5/2})$.  Under the same assumption on the growth of $q$, this was subsequently improved by Coppersmith and Pak~\cite{CopPak:00} to a mixing time of $O(n^2)$ which is optimal.  By taking $q$ so large, they can effectively assume that once an entry is updated to a non-zero entry it remains non-zero.

Our work resolves the case of fixed $q$ and also gives new bounds when $q$ is growing with $n$.
\begin{maintheorem}\label{t:generalQ}
There exists an absolute constant $C>0$ such that the mixing time of the lazy upper triangular walk $\mathcal{W}_1$ on $G_n(q)$ satisfies
\[
\tmix^{(n)} \leq C  n^2 \log q.
\]
\end{maintheorem}
This yields new bounds when $q < 2n^3$ and so is outside of the regime studied by Coppersmith and Pak.  As in the case $q=2$ there is a natural lower bound of order $n^2$.  We conjecture that the mixing time is order $n^2$ for all $q$ which interpolates between our fixed $q$ result and Coppersmith and Pak's large $q$ result.

Our proof crucially makes use of the linear structure of $G_n(q)$. When multiplying by $I+ a E_{n-1,n}$ we observe that the choice of $a$ only affects the final column.  Indeed we can write the off-diagonal portion of the final column at time $t$ as $b_0 + \sum_{k=1}^{J} a_k b_k$ for $b_k\in \Z_p^{n-1}$ where the $a_k\in \Z_p$ are uniformly chosen independent of the $b_k$ and the remainder of the matrix.  It follows that the final column will be uniform and  independent of the rest of the matrix if the vectors $\{b_1,\ldots,b_J\}$ span $\Z_p^{n-1}$.

We prove that for large enough $C$ and $t=C n^2 \log q$ that the $b_k$ do span with high probability by expressing the vectors as projections at random times of another matrix walk.  This walk is a reflection of the original walk involving column rather than row operations.  By the linear structure of the walk, the inner product of the a fixed vector and the projection of the Markov chain evolves as a site in the East model.  Then using known spectral gap estimates for the East model and large deviation results for Markov chains we obtain the necessary bounds on when the $b_k$ span $\Z_p^{n-1}$.  Iterating this argument over all the columns completes the proof.

\subsection{The East model}\label{s:East}
A key role is played by the one dimensional East model which is an extensively studied example of a kinetically constrained spin model  (see~\cite{AldDia:02,CMRT:08} and references therein).  Formally the East model with parameter $0<p<1$ is a Markov chain whose state space is
\[
\mathcal{H}_n=\{h=(h_1,\ldots,h_L)\in\{0,1\}^n:h_1=1\}.
\]
The dynamics are given as follows.  In each step choose a uniformly random $1\leq i \leq n-1$ and
\begin{itemize}
\item If $h_{i}=1$ set $h_{i+1}$ to 1 with probability $p$ and 0 with probability $1-p$.
\item  If  $h_{i}=0$ do nothing.
\end{itemize}
The dynamics are reversible with respect to the i.i.d. Bernoulli measure where the $h_i$ are 1 with probability $p$ for $2\leq i \leq n$.
When $q=2$ the East Model with $p=1/2$ describes the dynamics of a single column of the upper triangular matrix walk.

\section{Preliminaries}

It will be convenient for our proof to consider the continuous time version of $\mathcal{W}_1$ whose generator is given by
\begin{equation}\label{e:generator}
(\mathcal{L}f)(X)=\sum_{i=1}^{n-1} \frac1q \sum_{a\in\Z_q} f((I+ a E_{i,i+1})X)-f(X)
\end{equation}
This version admits the following description: for each $1\leq i \leq n-1$ we have a Poisson clock and when clock $i$ rings we choose a uniform $a\in \Z_q$ and add row $i+1$ multiplied by $a$ to row $i$.  Note that this speeds the chain up by a factor of $n-1$ of the original.  Let $X_t$ denote this Markov chain.  The continuous time representation has the useful property that for $1< n' < n$ the process $X_t$ projected onto the $n'\times n'$-submatrix of its first $n'$ rows and columns is itself a Markov chain whose law is given by the walk on $G_{n'}(q)$.

In this paper we consider mixing measured in the total variation distance.  For two probability measures $\nu_1,\nu_2$ on a countable space $\Omega$ the {\bf total variation distance} is defined as
\[
\|\nu_1-\nu_2\|_\tv = \max_{A\subset \Omega} |\nu_1(A)-\nu_2(A)| = \frac12\sum_{x\in\Omega} |\nu_1(x)-\nu_2(x)| ~.
\]
For an ergodic Markov chain $Y_t$ with stationary distribution $\nu$ we define
\[
\tmix(s) = \inf\{t: \max_{y\in \Omega} \|\P_y(Y_t\in \cdot) - \nu\|_\tv<s\} ~,
\]
and the (worst case) {\bf mixing time} is defined as $\tmix=\tmix(1/2e)$.  We denote the walk $\mathcal{W}_1$ on $G_n(q)$ with generator given in \eqref{e:generator} as $X_t=X_t^n$, its mixing time  as $\tmix^{(n)}$ and we let $d_n(t)$ denote $\|\P(X_t^n\in \cdot) - \pi^n\|_\tv$ where $\pi^n$ is the uniform distribution on $G_n(q)$ and the stationary distribution of $X_t$.

For an ergodic reversible Markov chain we denote the eigenvalues of the generator as $0=\lambda_1<\lambda_2\leq\ldots$ and the spectral gap, the second eigenvalue, as $\gap=\lambda_2$.
For any ergodic reversible finite Markov chain $\gap^{-1}\leq \tmix$ (c.f., e.g. \cites{AF,LPW}).
The mixing time of the lazy random walks and continuous time random walks are closely related.  Theorem 20.3 of~\cite{LPW} implies that if the mixing time of the continuous time random walk $\mathcal{W}_1$ is $O(n\log q)$ then the mixing time of the discrete time lazy version of the walk is $O(n^2 \log q)$.  As such for the rest of the paper we consider the continuous time version of the walk.

We will make use of Hoeffding type bounds for Markov chains.  There are various versions including results of Gillman~\cite{gillman:98} and sharper exponential rates given by Le{\'o}n and Perron~\cite{LeoPer:04}.
We will make use of Theorem~3.4 of~\cite{Lezaud:98} by Lezaud which gives a continuous time Markov chain analogue.  In particular it shows that for a reversible Markov chain $Y_t$ with stationary distribution $\nu$, and for any subset $A\subset \Omega$,
\begin{equation}\label{e:MCHoeffding}
\P\left(\left|\int_0^t I(Y_t \in A) dt  -t \nu(A)\right|>\epsilon t\right) \leq \frac{2}{\sqrt{\nu_{\min}}} \exp(-t\epsilon^2 \lambda/12)
\end{equation}
where $\nu_{\min}=\min\{\nu(x):x\in\Omega,\nu(x)>0\}$.

\subsection{Continuous time East models}\label{s:East}
We will also use the continuous time version of the East Model.
Its dynamics on $\mathcal{H}_n$ are given as follows.  For each $1\leq i \leq n-1$ we have a rate 1 Poisson clock.  When clock $i$ rings:
\begin{itemize}
\item If $h_{i}=1$ set $h_{i+1}$ to 1 with probability $p$ and 0 with probability $1-p$.
\item  If  $h_{i}=0$ do nothing.
\end{itemize}
Again the dynamics is reversible with respect to the i.i.d. Bernoulli measure where the $h_i$ are 1 with probability $p$ for $2\leq i \leq n$. Let $\theta_p(n)$ denote the spectral gap of the East model of length $n$ with parameter $p$.
Aldous and Diaconis~\cite{AldDia:02} showed that $\theta_p:=\inf_n\theta_p(n)>0$ for all $p$, establishing a bounded spectral gap for the infinite dynamics, and also giving the asymptotics as $p$ tends to 0.

We now consider a  natural $q$-state extension of the East model.  Define the state-space as
\[
\mathcal{H}_n^q=\{h=(h_1,\ldots,h_n)\in \Z_q^n:h_1=1\}
\]
and again for each $1\leq i \leq n-1$ we have rate 1 Poisson clocks so that when  clock $i$ rings:
\begin{itemize}
\item If $h_{i}\neq 0$ set $h_{i+1}$ to a uniform choice in $\Z_q$.
\item  If  $h_{i}=0$ do nothing.
\end{itemize}
Let $\lambda^q(n)$ denote the spectral gap of this process.
\begin{proposition}\label{p:SpectralGap}
The spectral gaps satisfy $\lambda^\star:=\inf_{q,n} \lambda^q(n)>0$.
\end{proposition}

\begin{proof}
Let $Z_t$ denote the $q$-state East model and let $\psi:\mathcal{H}^q_n\to \mathcal{H}_n$ given by
\[
\psi(h)_i = \begin{cases}
0 &\hbox{if } h_i=0,\\
1 &\hbox{otherwise.}
\end{cases}
\]
The process $Z_t'=\psi(Z_t)$ has the law of the standard East model with parameter $p=\frac{q-1}{q}$.

We now consider an update of $Z_t$ where clock $i$ rings and $Z_t(i)\neq 0$.  The site $i+1$ is updated to 0 with probability $\tfrac1q$ and from a uniform choice on $\{1,\ldots,q-1\}$ with probability $\tfrac{q-1}{q}$.  In particular it is uniformly distributed, conditional on the value $Z_t'(i+1)$.  It is easily verified that this conditional independence property holds at all times after such an update and is independent of the process at other sites.

So let $\mathcal{B}(t)$ be the event that for each $1\leq i \leq n-1$ there has been an update from clock $i$ at a time $s$ when $Z_s(i)\neq 0$ before time $t$.  It follows that on the event $\mathcal{B}(t)$ the variable $Z_t$ is uniformly distributed conditional on the value of $Z_t'$.  Hence we have that
\begin{align}\label{e:ZtTvDist}
\|\P(Z_t\in\cdot)-\P(Z\in\cdot)\|_{\tv} \leq \|\P(Z_t'\in\cdot)-\P(\psi(Z)\in\cdot)\|_{\tv} + \P(\mathcal{B}^c(t)) \, ,
\end{align}
where $Z$ is distributed according to the stationary distribution of $Z_t$.  Since the spectral gap of $Z_t'$ is $\theta_{\tfrac{q-1}{q}}(n)$ we have that
\begin{equation}\label{e:projectedVariationDist}
\|\P(Z_t'\in\cdot)-\P(\psi(Z)\in\cdot)\|_{\tv} \leq C(n,q) \exp(-t \theta_{\tfrac{q-1}{q}}(n)) \, ,
\end{equation}
for some constant $C(n,q)$ as the spectral gap gives the asymptotic convergence rate to stationarity.  Fix some $1\leq i \leq n-1$.  Then equation~\eqref{e:MCHoeffding} implies that for some $c>0$ and constant $C'(n,q)$
\begin{equation}
\P\left(\int_0^t I(Z_s(i)\neq 0) ds \leq \tfrac13 t\right) \leq C'(n,q) \exp(-c \theta_{\tfrac{q-1}{q}}(n) t).
\end{equation}
The times at which clock $i$ rings form a rate 1 Poisson process on $[0,t]$ which is independent of the process $Z_t(i)$.  Hence the chance that $\mathcal{B}$ fails is given by
\begin{align}\label{e:PsiCondIndep}
\P\left(\mathcal{B}^c(t)\right)&\leq \P\left(\exists 1\leq i \leq n-1, \int_0^t I(Z_t(i)\neq 0) dt\leq \tfrac13 t \right)\\
&\qquad+ (n-1)\P\left(\hbox{Poisson}(\tfrac13 t) = 0\right)\nonumber\\
&\quad \leq n C'(n,q) \exp(-c \theta_{\tfrac{q-1}{q}}(n) t) + n\exp(-\tfrac13 t).
\end{align}
Combining equations \eqref{e:ZtTvDist}, \eqref{e:projectedVariationDist} and \eqref{e:PsiCondIndep} we have
\begin{align*}
\|\P(Z_t\in\cdot)-\P(Z\in\cdot)\|_{\tv} &\leq C(n,q) \exp(-c \theta_{\tfrac{q-1}{q}}(n) t) \\
&\qquad + n C'(n,q) \exp(-c \theta_{\tfrac{q-1}{q}}(n) t) + n\exp(-\tfrac13 t)
\end{align*}
and so the exponential rate of decay is at least $c \theta_{\tfrac{q-1}{q}}\wedge \frac13$.
Hence we have that
\[
\lambda^q(n) \geq c \theta_{\tfrac{q-1}{q}}(n)\wedge \frac13.
\]
Theorem 4.2 of~\cite{CMRT:08} establishes that for some $0<p_0<1$, whenever $p_0<p<1$ we have that $\theta_p(n)>\frac12$.  Combining this with the fact that $\inf_n \theta_p(n)>0$ for all $p$ we have that
\[
\lambda^\star =\inf_{q,n} \lambda^q(n) \geq \frac13 \wedge \inf_{q,n} c \theta_{\tfrac{q-1}{q}}(n) > 0 \,
\]
which completes the proof.
\end{proof}

It was observed in~\cite{ADS:04}  that  the projection onto column $j$ of the upper triangular walk on $G_n(q)$ is given by an East model, specifically the $j$ length $q$-state East model.  Our proof makes use of this fact and additionally that the observation also holds for linear combinations of columns.

\section{Proof of Theorem~\ref{t:main}}

In this section we work exclusively with the continuous time version of the Markov chain.  We begin with some notation.  By symmetry we will assume that the Markov chain begins at $X_0=I$.  Fix some terminal time $T$.  First we split the clocks and associated row operations into those with $1\leq i \leq n-2$ and those with $i=n-1$.  Let $0<t_1<t_2<\ldots$ denote the times at which the Poisson clocks ring for  $1\leq i \leq n-2$ and let $W_j$ denote the associated row operation matrices.  We denote the $\sigma$-algebra generated by these updates as $\mathcal{F}$.
Let $N(t)=\max\{j \geq 0 : t_j \leq t\}$ and define the backwards process on the interval $[0,T]$ by $Y_0=I$ and
\[
Y_t=\prod_{j=0}^{N(T)-N(T-t)-1} W_{N(T)-j} = W_{N(T)} W_{N(T)-1}\ldots W_{N(T-t)+1}
\]
and for $0\leq t < t' \leq T$ we let
\[
Y_{t,t'} = Y_{t}^{-1} Y_{t'}= W_{N(T-t)} W_{N(T-t)-1}\ldots W_{N(T-t')+1}.
\]
The process $Y_t$ is a Markov process given by the following column dynamics description.  For each $1\leq i \leq n-2$ we have a Poisson clock and when clock $i$ rings we choose a uniform $a\in \Z_q$ and add column $i$ multiplied by $a$ to column $i+1$.  With this description column $n$  is fixed while the walk on the top left $(n-1)\times(n-1)$ submatrix is equivalent to the walk $\mathcal{W}_1$ on $G_{n-1}(q)$ up to a reflection of the matrix.

Next consider the second type of row operations which have $i=n-1$.  Let $s_1<s_2<\ldots$ denote the times at which row $n$ is added to row $n-1$ and let $a_1,a_2,\ldots$ denote the associated scalers. Note that these are independent of $\mathcal{F}$.  Denoting
\[
J(t)=\max\{j \geq 0 : s_j \leq t\},
\]
as the number of row operations with $i=n-1$ up to time $t$. The following expansions of $X_T$ allow us to track the contribution of operations of the second type, taking advantage of the linear nature of $G_n(q)$.
By construction, with the definitions above, we have that
\begin{equation}\label{e:XtRecursion}
X_{s_\ell}=(I+a_{\ell}E_{n-1,n}) Y_{T-s_{\ell},T-s_{\ell-1}} X_{s_{\ell-1}}
\end{equation}
for $1\leq \ell \leq J(T)$ where $s_0=0$.  We now show, by induction, that
\begin{equation}\label{e:XtInduction}
X_{s_\ell}=Y_{T-s_\ell,T} + \sum_{k=1}^{\ell} a_k Y_{T-s_\ell,T-s_k} E_{n-1,n}.
\end{equation}
For $\ell=0$ equation \eqref{e:XtInduction} is immediate and for any upper triangular matrix $Y\in G_n(q)$ we have that
\begin{equation}\label{e:YcommuteEn}
E_{n-1,n}Y=E_{n-1,n} \hbox{ and } E_{n-1,n}E_{n-1,n}=0.
\end{equation}
Then by equation \eqref{e:XtRecursion}
\begin{align*}
X_{s_{\ell+1}}&=(I+a_{\ell+1}E_{n-1,n}) Y_{T-s_{\ell+1},T-s_{\ell}} X_{s_{\ell}}\\
&=(I+a_{\ell+1}E_{n-1,n}) Y_{T-s_{\ell+1},T-s_{\ell}} \left(Y_{T-s_\ell,T} + \sum_{k=1}^{\ell} a_k Y_{T-s_\ell,T-s_k} E_{n-1,n}\right)\\
&=Y_{T-s_{\ell+1},T} + \sum_{k=1}^{\ell+1} a_k Y_{T-s_{\ell+1},T-s_k} E_{n-1,n}
\end{align*}
where the final equality follows by expanding the sum, applying equation~\eqref{e:YcommuteEn} and using the fact that $Y_{T-s_{\ell+1},T-s_{\ell+1}}=I$ completing the inductive step.  Finally, since $X_{T}=Y_{0,J(T)}X_{J(T)}$, it follows from equation~\eqref{e:XtInduction} that
\begin{equation}\label{e:XtExpansion}
X_T=Y_T + \sum_{k=1}^{J(T)} a_k Y_{T-s_k} E_{n-1,n}.
\end{equation}
This expression allows us to separate the mixing in the first $n-1$ columns from that of the final column.  Let $\mathcal{M}$ denote set of $n\times n$-matrices with entries in $\Z_p$ and let $\mathcal{M}^*$ denote the sub-space of matrices spanned by $E_{i,n}$ for $1\leq i \leq n-1$.
Each matrix $Y_{T-s_k} E_{n-1,n}$ is in $\mathcal{M}^*$.  Let $\mathcal{A}=\mathcal{A}(T,n)$ denote the event that the collection $\{Y_{T-s_k} E_{n-1,n}:1\leq k \leq J(T)\}$ spans $\mathcal{M}^*$.  Since the $a_k$ are independent of the $Y_{T-s_k}$ it follows that on $\mathcal{A}$ the sum $\sum_{k=1}^{J(T)} a_k Y_{T-s_k} E_{n-1,n}$ is uniformly distributed on $\mathcal{M}^*$ and so we have the following proposition.
\begin{proposition}\label{p:tmixRecursion}
The distances to stationarity satisfy $$d_n(T) \leq d_{n-1}(T) + \P(\mathcal{A}(T,n)^c).$$
\end{proposition}

\begin{proof}
Let $X$ be uniform on $G_n(q)$, the stationary distribution of $X_t$.  As observed above, the Markov chain restricted to the first $n-1$ columns of $X_t$ gives the walk $\mathcal{W}_1$ on $G_{n-1}(q)$.  Hence we may couple the first $n-1$ columns of $X_T$ and $X$ except with probability $d_{n-1}(T)$.  Moreover the walk on the first $n-1$ columns is $\mathcal{F}$ measurable and independent of the scalars $\{a_1,\ldots,a_{J(T)}\}$.

On the event $\mathcal{A}(T,n)$ we have that $\sum_{k=1}^{J(T)} a_k Y_{T-s_k} E_{n-1,n}$ is uniformly distributed on $\mathcal{M}^*$ and independent of $\mathcal{F}$. Hence we may couple the final column of $X_T$ and $X$ independent  of $\mathcal{F}$ when $\mathcal{A}(T,n)$ holds.  Hence we may couple $X_T$ and $X$ except with probability $d_{n-1}(T) + \P(\mathcal{A}(T,n)^c)$ which completes the proposition.
\end{proof}

It remains to bound the probability of $\mathcal{A}$.  For this we use the fact that $Y_t$ is a Markov chain and its first $n-1$ columns are given by the column version of the walk $\mathcal{W}_1$.

\begin{proposition}\label{p:Span}
There exist absolute constants $0< c_1,c_2 < \infty$ such that $\P(\mathcal{A}(T,n)) \geq 1- e^{-c_1 (T-c_2 n\log q)}$.
\end{proposition}

\begin{proof}
The vectors $\{Y_{T-s_k} E_{n-1,n}:1\leq k \leq J(T)\}$ span $\mathcal{M}^*$ if and only if for each  non-zero vector $b=(b_1,\ldots,b_{n-1},0)\in \Z_q^{n}\setminus \{(0,\ldots,0)\}$ we have that for some $1 \leq k \leq J(T)$ that
\begin{equation}\label{e:spanRequirement}
b \cdot Y_{T-s_k} e_{n-1} \neq 0.
\end{equation}
Let $l=\min_j b_j\neq 0$.  We may assume without loss of generality that $b_l=1$ since multiplying $b$ by a non-zero scaler does not affect \eqref{e:spanRequirement}.

Now recall our observation that $Y_t$ is the dynamics on upper triangular matrices where for each $1\leq i \leq n-2$  we add column $i$ times a random scaler $a$ to column $i+1$ according to the times of rate 1 Poisson clocks.  The process $Z_t=b \cdot Y_t$ is also a Markov chain on vectors of length $n$ with the following description. For each $1\leq i \leq n-2$ at rate 1 we add position $i$ times a random scaler $a$ to position $i+1$.  Its initial condition is $Z_0=b$ and so the value of its first $l-1$ entries and the final entry are fixed as 0.  The value in entry $l$ is fixed at 1 and entries $l$ to $n-1$ perform the $q$-state East model of length $n-l$.  By Proposition~\ref{p:SpectralGap} this chain has spectral gap at least $\lambda^\star>0$.  Its state space is the uniform distribution over $q^{n-l-1}$ possible vectors and so applying equation~\eqref{s:East} we have that
\begin{equation}\label{e:walkLemma}
\P\left(\int_0^T I(Z_t e_{n-1}\neq 0) dt \leq \tfrac13 T\right) \leq q^{n/2}  \exp(-c \lambda^\star T)
\end{equation}
where $I(\cdot)$ denotes the indicator and $c$ is an absolute constant.
The times $T-S_k$ form a rate 1 Poisson process on $[0,T]$ which is independent of the process $Y_t$.  Hence
\begin{align}\label{e:walkLemma2}
&\P\left(\forall 1 \leq k \leq J(T), Z_{T-s_k} e_{n-1} = 0\right)\nonumber\\
&\quad\leq \P\left(\int_0^T I(Z_t e_{n-1}\neq 0) dt\leq \tfrac13 T \right)
+ \P\left(\hbox{Poisson}(\tfrac13 T) = 0\right)\nonumber\\
&\quad \leq q^{n/2}  \exp(-c \lambda^\star T) + \exp(-\tfrac13 T).
\end{align}
Taking a union bound over $b$ completes the proof.
\end{proof}

We are now ready to prove Theorem~\ref{t:generalQ} which in turn proves Theorem~\ref{t:main}

\begin{proof}[Proof of Theorem~\ref{t:generalQ}]
As noted above it is sufficient by Theorem 20.3 of~\cite{LPW} to prove an upper bound on the mixing time of the continuous time chain~$\mathcal{W}_1$ is $O(n\log q)$.
Combining Propositions~\ref{p:tmixRecursion} and~\ref{p:Span} we have that
\[
d_n(T) \leq d_{n-1}(T) + e^{-c_1 (T-c_2 n\log q)}
\]
and hence by induction
\[
d_n(T)=\sum_{i=2}^n e^{-c_1 (T-c_2 i\log q)}.
\]
For $c>c_2$ we have that $d_n(c n\log q)=o(1)$ which completes the proof.
\end{proof}

\subsection*{Acknowledgements}
The authors would like to thank Persi Diaconis for suggesting the problem and for illuminating discussions and Jian Ding and Weiyang Ning for comments on an earlier draft.


\begin{bibdiv}
\begin{biblist}

\bib{AldDia:02}{article}{
  title={{The asymmetric one-dimensional constrained Ising model: rigorous results}},
  author={Aldous, D.},
  author={Diaconis, P.},
  journal={Journal of Statistical Physics},
  volume={107},
  pages={945--975},
  date={2002},
}




\bib{AF}{book}{
    AUTHOR = {Aldous, David},
    AUTHOR = {Fill, James Allen},
    TITLE =  {Reversible {M}arkov Chains and Random Walks on Graphs},
    note = {In preparation, \texttt{http://www.stat.berkeley.edu/\~{}aldous/RWG/book.html}},
}

\bib{ADS:04}{article}{
  title={{A super-class walk on upper-triangular matrices}},
  author={Arias-Castro, E.},
  author={Diaconis, P.},
  author={Stanley, R.},
  journal={Journal of algebra},
  volume={278},
  pages={739--765},
  date={2004},
}


\bib{CMRT:08}{article}{
  title={{Kinetically constrained spin models}},
  author={Cancrini, N.},
  author={Martinelli, F.},
  author={Roberto, C.},
  author={Toninelli, C.},
  journal={Probability Theory and Related Fields},
  volume={140},
  pages={459--504},
  year={2008},
}





\bib{CopPak:00}{article}{
  title={{Random walk on upper triangular matrices mixes rapidly}},
  author={Coppersmith, D.},
  author={Pak, I.},
  journal={Probability Theory and Related Fields},
  volume={117},
  pages={407--417},
  date={2000},
}

\bib{Diaconis:03}{article}{
  title={{Random walks on groups: characters and geometry}},
  author={Diaconis, P.},
   conference={
      title={Groups Saint Andrews 2001 in Oxford},
   },
  pages={120--142},
  year={2003},
  publisher={Cambridge Univ. Pr.}
}



\bib{DiaSal:94}{article}{
  title={{Moderate growth and random walk on finite groups}},
  author={Diaconis, P.},
  author={Saloff-Coste, L.},
  journal={Geometric and Functional Analysis},
  volume={4},
  pages={1--36},
  date={1994},
}

\bib{DiaSal:95}{article}{
  title={{An Application of Harnaek Inequalities to Random Walk on Nilpotent Quotients}},
  author={Diaconis, P.},
  author={Saloff-Coste, L.},
   conference={
      title={Proceedings of the Conference in Honor of Jean-Pierre Kahane},
   },
  pages={189},
  date={1995},
}


\bib{DiaSal:96}{article}{
  title={{Nash inequalities for finite Markov chains}},
  author={Diaconis, P.},
  author={Saloff-Coste, L.},
  journal={Journal of Theoretical Probability},
  volume={9},
  pages={459--510},
  date={1996},
}


\bib{DiaSha:81}{article}{
  title={{Generating a random permutation with random transpositions}},
  author={Diaconis, P.}
  author={Shahshahani, M.},
  journal={Probability Theory and Related Fields},
  volume={57},
  pages={159--179},
  date={1981},
}




\bib{Ellenberg:93}{thesis}{
  AUTHOR =       {Ellenberg, J.},
  TITLE =        {{A Sharp Diameter Bound for Upper Triangular Matrices}},
  SCHOOL =       {Senior Honors Thesis, Department of Mathematics, Harvard University},
  YEAR =         {1993},
}

\bib{gillman:98}{article}{
  title={{A Chernoff Bound for Random Walks on Expander Graphs}},
  author={Gillman, D.},
  journal={SIAM Journal on Computing},
  volume={27},
  pages={1203},
  date={1998}
}





\bib{LeoPer:04}{article}{
  title={{Optimal Hoeffding bounds for discrete reversible Markov chains}},
  author={Le{\'o}n, C.A.},
  author={Perron, F.},
  journal={Annals of Applied Probability},
  volume={14},
  pages={958--970},
  date={2004},
  publisher={JSTOR}
}




\bib{LPW}{book}{
  title={{Markov chains and mixing times}},
  author={Levin, D.A.},
  author={Peres, Y.},
  author={Wilmer, E.L.},
  journal={American Mathematical Society},
  date={2008},
}


\bib{Lezaud:98}{article}{
  title={{Chernoff-type bound for finite Markov chains}},
  author={Lezaud, P.},
  journal={The Annals of Applied Probability},
  volume={8},
  pages={849--867},
  date={1998},
}




\bib{Pak:00}{article}{
  title={{Two random walks on upper triangular matrices}},
  author={Pak, I.},
  journal={Journal of Theoretical Probability},
  volume={13},
  pages={1083--1100},
  date={2000},
}


\bib{SaloffCoste:96}{article}{
   author={Saloff-Coste, Laurent},
   title={Lectures on finite Markov chains},
   conference={
      title={Lectures on probability theory and statistics},
      address={Saint-Flour},
      date={1996},
   },
   book={
      series={Lecture Notes in Math.},
      volume={1665},
      publisher={Springer},
      place={Berlin},
   },
   date={1997},
   pages={301--413},
}




\bib{Stong:95}{article}{
  title={{Random walks on the groups of upper triangular matrices}},
  author={Stong, R.},
  journal={The Annals of Probability},
  volume={23},
  pages={1939--1949},
  date={1995},
  publisher={JSTOR}
}



\bib{Zack:84}{thesis}{
  AUTHOR =       {Zack, M.},
  TITLE =        {{A Random Walk on the Heisenberg Group}},
  SCHOOL =       {University of California, San Diego},
  YEAR =         {1984},
}


\end{biblist}
\end{bibdiv}
\end{document}